\documentclass[11pt]{article} 

\usepackage[english]{babel}

\usepackage[letterpaper,top=3.2cm,bottom=3.2cm,left=3.0cm,right=3.0cm,marginparwidth=1.75cm]{geometry}

\usepackage{amsmath}
\usepackage{graphicx}
\usepackage[colorlinks=true, allcolors=blue]{hyperref}
\usepackage{amssymb}
\usepackage{amsthm}
\usepackage{esint}
\usepackage{enumitem}
\usepackage{titling}  
\usepackage{authblk} 

\numberwithin{equation}{section}

\newtheoremstyle{mydefinitionstyle}
  {2em}
  {2em}
  {}
  {}
  {\scshape}
  {.}
  {0.5em}
  {}

\theoremstyle{mydefinitionstyle}
\newtheorem{definition}{Definition}

\newtheoremstyle{mytheoremstyle}
  {2em}
  {2em}
  {\itshape}
  {}
  {\bfseries}
  {.}
  {0.5em}
  {}

\theoremstyle{mytheoremstyle}
\newtheorem{theorem}{Theorem}[section]
\newtheorem{lemma}[theorem]{Lemma}
\newtheorem{proposition}[theorem]{Proposition}
\newtheorem{corollary}[theorem]{Corollary}

\newtheoremstyle{myremarkstyle}
  {1.6em}
  {1.6em}
  {}
  {}
  {\scshape}
  {.}
  {0.5em}
  {}

\theoremstyle{myremarkstyle}
\newtheorem{remark}{Remark}[section]

\newcommand\R{\mathbb{R}}

\title{The Sharp Measure Upper Bound of the Nodal Sets of Neumann Laplace Eigenfunctions on $C^{1,1}$ Domains}

\author[1]{Xiujin Chen}
\author[2]{Xiaoping Yang}

\affil[1]{\small School of Mathematics, Nanjing University, Nanjing 210093, China, \texttt{xiujin\_chen@smail.nju.edu.cn}}
\affil[2]{\small School of Mathematics, Nanjing University, Nanjing 210093, China, \texttt{xpyang@nju.edu.cn}}

\date{}

\begin{document}

\maketitle

\begin{abstract}
     Let $\Omega$ be a bounded domain in $\R^n$ with $C^{1,1}$ boundary and let $u_\lambda$ be a Neumann Laplace eigenfunction in $\Omega$ with eigenvalue $\lambda$. We show that the $(n-1)$-dimensional Hausdorff measure of the zero set of $u_\lambda$ does not exceed $C\sqrt{\lambda}$.
\end{abstract}

\begin{section}{Introduction}

In this paper, our goal is to obtain the sharp upper bound for the area of the nodal sets of Neumann Laplace eigenfunctions in $C^{1,1}$ domains. The nodal set of $u$ is denoted by $Z(u)=\{ x\in\Omega : u(x)=0 \}$. 

Throughout this paper, we always use the notation $\nu$ to indicate the unit outer normal vector field on the boundary in the context. With this convention, our main result is stated as follows.

\begin{theorem}\label{main thm}
    Let $\Omega$ be a bounded $C^{1,1}$ domain and let $u_\lambda$ satisfy $\Delta u_\lambda + \lambda u_\lambda = 0$ in $\Omega$ and $\frac{\partial u_\lambda}{\partial \nu}|_{\partial \Omega} = 0$. Then
    \begin{equation*}
        \mathcal{H}^{n-1} (Z(u_\lambda)) \leq C\sqrt{\lambda},
    \end{equation*}
    where $C$ depends only on $\Omega$.
\end{theorem}

The argument we adopt mainly follows the spirit of Logunov, Malinnikova, Nadirashvili and Nazarov's work \cite{logunov2021sharp}, where they studied the nodal sets of Dirichlet eigenfunctions. But for the Neumann problem, the obstacles are very different especially near the boundary. We also improve the argument in \cite{logunov2021sharp} when we consider a boundary neighborhood with the large doubling index. Beside that, the main novelty of this paper is in Section \ref{Approximation for the Neumann Problem}, where we construct an approximation for the Neumann problem by delicately applying the maximum principle. 

We briefly introduce the history of the study on the nodal sets of Laplace eigenfunctions. Let $\Delta_M$ be the Laplace operator on a 
$n$-dimensional smooth compact manifold $M$ and $u_\lambda$ be a corresponding eigenfunction, $\Delta_M u_\lambda + \lambda u_\lambda = 0$ on $M$. S. T. Yau \cite{yau1982seminar} conjectured that the surface area of the nodal set of $u_\lambda$ satisfies
\begin{equation}\label{yau's conj}
    c(M)\sqrt{\lambda} \leq \mathcal{H}^{n-1}(Z(u_\lambda)) \leq C(M)\sqrt{\lambda}.
\end{equation}
Donelly and Fefferman \cite{donnelly1988nodal} proved the conjecture for manifolds with analytic metrics. For smooth manifolds, the lower bound and a polynomial upper bound were obtained by Logunov in \cite{logunov2018nodallower} and \cite{logunov2018nodalupper} respectively. For other related works, see the references in \cite{logunov2018nodalupper,logunov2018nodallower}.

For boundary value problems, Donelly and Fefferman \cite{donnelly1990nodal} also proved the estimate \eqref{yau's conj} for manifolds with analytic boundaries assuming that $u_\lambda$ satisfies either Dirichlet or Neumann boundary condition. Kukavica \cite{kukavica1995nodal} generalized their results to eigenfunctions of elliptic operators with real analytic coefficients. Liu, Tian and Yang \cite{liu2024measure} obtained a measure upper bound for Robin eigenfunctions when the boundary is piecewise analytic. Lin and Zhu \cite{lin2022upper} proved the similar estimates to \eqref{yau's conj} for bi-Laplacian eigenfunctions with the assumption that the boundary is real analytic. Tian and Yang \cite{tian2022measure} obtained the sharp upper bound for bi-Laplacian eigenfunctions when the boundaries are piecewise analytic. 

For domains that are not analytic, there are less results. The polynomial upper bounds for the area of the zero set of Robin Laplace eigenfunctions in smooth domains was obtained by Zhu \cite{zhu2018nodal}. The sharp upper bound for the area of the zero set of Dirichlet Laplace eigenfunctions in $C^1$ domains or Lipschitz domains with small Lipschitz constants was obtained by Logunov, Malinnikova, Nadirashvili and Narazov \cite{logunov2021sharp}.

This paper is organized as follows. In Section \ref{Preliminaries}, we introduce the definition of Lipschitz domains with small Lipschitz constants and list several useful lemmas. In Section \ref{Doubling Index}, we introduce facts about the doubling index, especially the almost monotonicity. In Section \ref{Approximation for the Neumann Problem}, we construct the standard approximation to the Neumann problem near the boundary. In Section \ref{Nodal Sets near the Boundary}, we use the approximation in Section \ref{Approximation for the Neumann Problem} to study the nodal sets near the boundary. In Section \ref{Proof of Theorem}, we complete the proof of Theorem \ref{nodal set bounds for harmonic functions}, which is a measure upper bound of nodal sets for harmonic functions. In Section \ref{Neumann Laplace Eigenfunctions}, we introduce the standard extension method for Laplace eigenfunctions and use Theorem \ref{nodal set bounds for harmonic functions} to complete the proof of the main theorem.

\end{section}

\begin{section}{Preliminaries}\label{Preliminaries}

\begin{subsection}{Lipschitz domains with small Lipschitz constants}

Throughout this paper, for any $x\in\R^n$, we use the notation $x=(x',x'')\in\R^{n-1}\times\R$ where $x'$ is the first $n-1$ coordinates of $x$ and $x''$ is the last coordinate of $x$.

\begin{definition}\label{def local tau Lip}
    Let $\Omega$ be a domain in $\R^n$, $\tau\in(0,1)$, and let $B=B(x,r)$ for some $x\in\partial\Omega$. We say that $\partial\Omega\cap B$ is $\tau$-Lipschitz if there is an isometry $T$: $\R^n\to\R^n$ and a Lipschitz function $f$: $B^{n-1}(0,r)\to\R$ with the Lipschitz constant bounded by $\tau$ such that $T(0)=x$, $f(0)=0$ and
    \begin{equation*}
        \Omega \cap B = T \left( \{ (y',y'')\in B^n(0,r) \subset \R^{n-1} \times \R: y''>f(y') \} \right).
    \end{equation*}
    In this case we write $\partial\Omega\cap B\in Lip(\tau)$.
\end{definition}

\begin{remark}\label{choice of the coordinate}
    In this paper, most of the arguments are local. When considering the part of the boundary $\partial\Omega\cap B(x,r)$, we choose the local coordinate system in $B(x,r)$ such that $x=0$ and $e_n=-\nu(x)$, where $e_n$ is the unit vector in the direction of the last coordinate.
\end{remark}

\begin{definition}
    We say that $\Omega$ is a Lipschitz domain with local Lipschitz constant $\tau$ if there exists $r>0$ such that $\partial \Omega \cap B(x,r) \in Lip(\tau)$ for any $x \in \partial \Omega$.
\end{definition}

\begin{remark}
    Clearly, any bounded $C^1$ domain is a Lipschitz domain with local Lipschitz constant for any positive $\tau$. Of course so is any bounded $C^{1,1}$ domain. 
\end{remark}
    
\end{subsection}

\begin{subsection}{Auxiliary lemmas}

We list several useful lemmas that will be repeatedly used in the following sections. The first one is the local boundedness of solutions to the Neumann problem.

\begin{lemma}\label{local boundedness}
    Let $\Omega$ be a bounded Lipschitz domain, $h\in C(\overline{\Omega})$, $\Delta u=0$ in $\Omega$, and $\frac{\partial h}{\partial \nu}|_{\partial \Omega} = 0$. Then there exists $r_0>0$ such that for any $x\in\overline{\Omega}$ and $0<r<r_0$,
    \begin{equation*}
        \sup_{B(x,r)\cap\Omega} h^2 \leq \frac{C}{|B(x,2r)|} \int_{B(x,2r)\cap\Omega} h^2,
    \end{equation*}
    where $C$ depends only on the Lipschitz character of $\Omega$.
\end{lemma}

\begin{proof}
    The standard Moser method. See, e.g., Lemma 3.1 in \cite{lanzani2005robin}.
\end{proof}

\begin{lemma}\label{local gradient estimate}
    Let $\Omega$ be a bounded $C^{1,1}$ domain, $h\in C^1(\overline{\Omega})$, $\Delta u=0$ in $\Omega$, and $\frac{\partial h}{\partial \nu}|_{\partial \Omega} = 0$. Then there exists $r_0>0$ such that for any $x\in\overline{\Omega}$ and $0<r<r_0$,
    \begin{equation*}
        \sup_{B(x,r)\cap\overline{\Omega}} |\nabla h| \leq \frac{C}{r} \sup_{B(x,2r)\cap\Omega} |h|,
    \end{equation*}
    where $C$ depends only on the $C^{1,1}$ character of $\Omega$.
\end{lemma}

\begin{proof}
     By a standard $C^{1,1}$ coordinate transformation $\Psi$ (See e.g. \cite{adolfsson1997c1}), the Neumann problem in a neighborhood of $\partial\Omega$ can be transformed to the Neumann problem in a unit half ball $B_+=\{x\in\R^n:|x|<1,x''>0\}$. Let $v(x)=u\circ\Psi(x)$, we have $-\mathrm{div}(A\nabla v)= 0$ in $B_+$ and $\frac{\partial v}{\partial \nu}=0$ on the flat boundary, where $A=\mathrm{det}J\Psi J\Psi^{-T}J\Psi^{-1}$ and $J\Psi$ denotes the Jacobi matrix of $\Psi$. The fact that $\Psi$ is a $C^{1,1}$ transformation implies $A$ is Lipschitz. By extending $v$ to the whole ball symmetrically and the interior gradient estimates, we have for sufficiently small $r$,
    \begin{equation*}
        \sup_{\Psi^{-1}(B(x,r)\cap\overline{\Omega})} |\nabla v| \leq \frac{C}{r} \sup_{\Psi^{-1}(B(x,2r)\cap\Omega)} |v|,
    \end{equation*}
    where $C$ depends on the transformation $\Psi$ and thus on the $C^{1,1}$ character of $\partial\Omega$. Back to the original domain, since $\Psi$ is non-degenerate, we obtain the required estimate.
\end{proof}

The following smallness propagation lemma is from \cite{alessandrini2009stability}.

\begin{lemma}\label{propagation of smallness}
    Let $B_+$ be the half-ball, $B_+=\{ (x',x'')\in \R^{n-1}\times\R: |x'|^2+|x''|^2<1, x''>0 \}$, and $\Gamma$ be the flat part of $\partial B_+$, $\Gamma = \{ (x',x'')\in\overline{B_+}:x''=0 \}$. There exist $\gamma=\gamma(n)\in(0,1)$ and $C=C(n)>0$ such that if $h$ is harmonic in $B_+$, $h\in C^1(\overline{B_+})$ and satisfies the inequalities 
    \begin{equation*}
        |h|\leq 1, \ |\nabla h|\leq 1 \ \text{in } B_+, \quad |h|\leq \varepsilon, \ |\partial_n h|\leq \varepsilon \ \text{on } \Gamma,
    \end{equation*}
     then
     \begin{equation*}
         |h(x)| \leq C \varepsilon^\gamma \ \text{in } \frac{1}{3}B_+,
     \end{equation*}
    where $\frac{1}{3}B_+=\{ (x',x''):|x'|^2+|x''|^2=\frac{1}{9},x''>0 \}$.
\end{lemma}
    
\end{subsection}

\end{section}

\begin{section}{Doubling Index}\label{Doubling Index}

Let $h\in C(\overline{\Omega})$ be a non-zero harmonic function in a domain $\Omega\subset\R^n$. For each $x\in\overline{\Omega}$ and $r>0$, let
\begin{equation*}
    H_h(x,r) = \int_{B(x,r)\cap\Omega} h^2 \quad \text{and} \quad N_h(x,r)=\log\frac{H_h(x,2r)}{H_h(x,r)}.
\end{equation*}
The quantity $N_h(x,r)$ is called the doubling index of $h$ in $B(x,r)$. For $B(x,2R)\subset\Omega$, it is a well-known fact that for any $0<r<R$,
\begin{equation}\label{monotonicity of doubling interior}
    N_h(x,r)\leq N_h(x,R).
\end{equation}
 For an elementary proof, see, e.g., \cite{korevaar1994logarithmic}.
 
 For a cube $Q$ with diameter $l$ such that $Q\cap\Omega\neq\varnothing$, define the maximal doubling index of $h$ in $Q$ by
\begin{equation*}
    N^*_h(Q) = \sup_{x\in Q\cap\overline{\Omega},0<r\leq l} N_h(x,r).
\end{equation*}
By this definition, we have for any smaller cube $q\subset Q$, there holds $N^*_h(q)\leq N^*_h(Q)$.

The following lemma is called the almost monotonicity of the doubling index. Comparing to \eqref{monotonicity of doubling interior}, in the following statement, the doubling index near the boundary is considered.

\begin{lemma}\label{monotonicity of doubling}
    Let $\Omega$ be a $C^{1,1}$ domain in $\R^n$, $B=B(x_0,R)$ for some $x_0\in\partial\Omega$, $\Delta h=0$ in $\Omega\cap B$ and $\frac{\partial h}{\partial \nu} = 0$ on $\partial \Omega\cap B$, then
    \begin{equation*}
        N_h(x,r_1) \leq C (N_h(x,r_2)+1),
    \end{equation*}
    for any $x\in \overline{\Omega}\cap\frac{1}{4}B$ and $0<2r_1\leq r_2<R/16$, where $C$ depends on the $C^{1,1}$ character of $\partial\Omega\cap B$.
\end{lemma}

For the proof of Lemma \ref{monotonicity of doubling}, we refer the readers to \cite{zongyuan2023quantitative}, where the more general Robin problem is considered. Also notice that we consider the Laplace equation with Neumann boundary value in this lemma, so better monotonicity of the doubling index can be obtained and Lemma \ref{monotonicity of doubling} follows.

By this kind of almost monotonicity, a simple corollary is the three-ball estimate.

\begin{corollary}\label{three ball ineq}
    Let $\Omega$ be a $C^{1,1}$ domain in $\R^n$, $B=B(x_0,R)$ for some $x_0\in\partial\Omega$, $\Delta h=0$ in $\Omega\cap B$ and $\frac{\partial h}{\partial \nu} = 0$ on $\partial \Omega\cap B$, then
    \begin{equation*}
        \sup_{B(x,\frac{3}{2}r)\cap\Omega} |h| \leq C \left( \sup_{B(x,r)\cap\Omega} |h| \right)^{1-\delta} \left( \sup_{B(x,4r)\cap\Omega} |h| \right)^\delta,
    \end{equation*}
    for any $x\in \overline{\Omega}\cap\frac{1}{4}B$ and $0<2r_1\leq r_2<R/16$, where $C$ and $\delta$ depend on the $C^{1,1}$ character of $\partial\Omega\cap B$.
\end{corollary}

\begin{proof}
    In Lemma \ref{monotonicity of doubling}, take $r_1=r$ and $r_2=2r$ and take exponents from both sides, we obtain
    \begin{equation*}
        H_h(2r) \leq C H_h(r)^{1-\delta} H_h(4r)^\delta.
    \end{equation*}
    By Lemma \ref{local boundedness}, we obtain the similar estimate in terms of $\sup |h|$ as required.
\end{proof}

It is a standard way to extend the Laplace eigenfunction to a harmonic function. See Section \ref{Neumann Laplace Eigenfunctions} for details. In the following sections, we focus on the nodal sets of harmonic functions (with Neumann boundary value). For harmonic functions with Neumann boundary value, we claim the following estimate on the nodal set.

\begin{theorem}\label{nodal set bounds for harmonic functions}
    There exists absolute constants $\tau_n>0$ and $C>0$ such that the following holds. Let $\Omega$ be a bounded $C^{1,1}$ domain, $\partial\Omega\cap B(x,128r)\in Lip(\tau)$ for some $x\in\partial\Omega$ and $\tau<\tau_n$. Then
    \begin{equation*}
        \mathcal{H}^{n-1}(Z(h)\cap B(x,r)) \leq C(N_h(x,4r)+1)r^{n-1},
    \end{equation*}
    for any non-zero harmonic function $h\in C(\overline{\Omega})$ satisfying $\frac{\partial h}{\partial \nu}=0$ on $\partial\Omega\cap B(x,128r)$.
\end{theorem}

The fact that Theorem \ref{nodal set bounds for harmonic functions} implies Theorem \ref{main thm} will be shown in Section \ref{Neumann Laplace Eigenfunctions}. Before that, we are devoted to proving Theorem \ref{nodal set bounds for harmonic functions}. We also need the estimate for the zero set inside the domain, which was proved by Donnelly and Fefferman in \cite{donnelly1988nodal}.

\begin{lemma}\label{nodal set bounds for interior balls}
    Let $h$ be a non-zero harmonic function in $\Omega\subset\R^n$. There exists $C$ such that
    \begin{equation*}
        \mathcal{H}^{n-1}(Z(h)\cap B) \leq C(N_h(x,4r)+1)r^{n-1},
    \end{equation*}
    for any ball $B=B(x,r)$ satisfying $\overline{B(x,8r)}\subset\Omega$.
\end{lemma}

\end{section}

\begin{section}{Approximation for the Neumann Problem}\label{Approximation for the Neumann Problem}

In this section, we introduce how to approximate the Neumann problem near the $C^{1,1}$ boundary by a mixed boundary value problem in a regular domain, in fact, a cylinder.

For any $x\in\R^n$, we use the notation $x=(x',x'')\in\R^{n-1}\times\R$ where $x'$ is the first $n-1$ coordinates of $x$ and $x''$ is the last coordinate of $x$. We also let $S(x,r,h):=\{y\in\R^n:|y'-x'|<r,0<y''-x''<h\}$ denote the cylinder with radius $r$ and height $h$ whose lower surface is centred at $x$.

\begin{definition}
    Let $\Omega$ be a bounded $C^{1,1}$ domain, $x_0\in\partial\Omega$ and $\partial\Omega\cap B(x_0,10)\in Lip(\tau)$. Choose the coordinate as in Remark \ref{choice of the coordinate} so $x_0=0$. Let $x_1=x_0-3\tau e_n$ and $S=S(x_1,1,1)$. For $\tau<1/100$, we have $\Omega\cap S\neq\varnothing$, the upper surface of $S$ is contained in $\Omega$ and the lower surface of $S$ is contained in $\R^n\setminus\Omega$. We call the triple $(x_1,S,\Omega\cap S)$ \textit{the standard approximation} near $x_0$. We remark that this approximation depends on $\tau$.
\end{definition}

We have introduced how to find a cylinder to approximate to a $\tau$-Lipschitz domain in a neighborhood of the boundary. Next we find a solution to the Neumann problem in the cylinder to approximate to the solution to the Neumann problem in the neighborhood.

\begin{theorem}\label{thm approximation}
    Let $\Omega$ be a bounded $C^{1,1}$ domain, $x_0\in\partial\Omega$ and $\partial\Omega\cap B(x_0,10)\in Lip(\tau)$ with $\tau<1/100$. Also let $h\in C(\overline{\Omega})$ be normalized as $\sup_{\Omega\cap B(x_0,3)}|h|=1$ and
    \begin{equation*}
        \left\{
        \begin{aligned}
            & \Delta h = 0 \ \text{in}\ \Omega \\
            & \frac{\partial h}{\partial \nu} = 0 \ \text{on}\ \partial\Omega\cap B(x_0,10) \\
        \end{aligned}
        \right.
    \end{equation*}
    Consider the standard approximation $(x_1,S,\Omega\cap S)$ near $x_0$. Then there exists $g\in C(\overline{S})$ such that
    \begin{equation*}
        \left\{
        \begin{aligned}
            & \Delta g = 0 \ \text{in}\ S \\
            & \frac{\partial g}{\partial e_n} = 0 \ \text{on}\ \Gamma_- \\
        \end{aligned}
        \right.
    \end{equation*}
    where $\Gamma_-:=\{(x',x'')\in \overline{S(x_1,1,1)}: x''=x_1''\}$ is the lower surface of $S$, and
    \begin{equation*}
        \sup_{\Omega\cap S} |g-h| < C\tau, \quad \sup_S |g| = \sup_{\Omega\cap S}|h|.
    \end{equation*}
\end{theorem}

To prove this theorem, we need the following quantitative uniqueness of a mixed valued problem in a domain close to a cylinder.

\begin{proposition}\label{quantitative uniqueness of a mixed boundary value problem}
    Let $S=S(0,1,1)\subset\R^n$ be a cylinder, $\Sigma=\{(x',x''))\in\partial S:0<x''<1\}$ be its side surface, $\Gamma_+=\{(x',x''))\in\partial S:x''=1\}$ be its upper surface, $f:\R^{n-1}\to\R$ be a function such that $0<f(x')<1/3$ for any $x'\in B^{n-1}(0,1)$, $\Omega=\{(x',x'')\in\R^n:f(x')>x''\}$ and $w\in C(\overline{S\cap\Omega})$ be a solution to the following problem:    
    \begin{equation*}
        \left\{
        \begin{aligned}
            & \Delta w = 0 \ \text{in}\ S\cap\Omega \\
            & w=0 \ \text{on}\ \Sigma\cap\Omega \\
            & \left| \frac{\partial w}{\partial e_n} \right| < \delta \ \text{on}\ \Gamma_+\cup(\partial\Omega\cap S)\\
        \end{aligned}
        \right.
    \end{equation*}
    Then
    \begin{equation*}
        \sup_{\Omega\cap S}|w|<C(n)\delta.
    \end{equation*}
\end{proposition}

\begin{proof}
    By normalization, it is sufficient to show that for any solution to the following problem:
    \begin{equation*}
        \left\{
        \begin{aligned}
            & \Delta w = 0 \ \text{in}\ S\cap\Omega \\
            & w=0 \ \text{on}\ \Sigma\cap\Omega \\
            & \sup_{\Omega\cap S} |w| = 1\\
        \end{aligned}
        \right.
    \end{equation*}
    there holds
    \begin{equation}\label{ineq1}
        \sup_{\Gamma_+\cup(\partial\Omega\cap S)} \left| \frac{\partial w}{\partial e_n} \right| \geq c(n).
    \end{equation}
    By maximum principle, $\sup_{\Omega\cap S}|w|$ can only be attained at some $\overline{x}\in\Gamma_+\cup(\partial\Omega\cap S)$. Without loss of generality, we assume $w(\overline{x})>0$ and thus $\sup_{\Gamma_+\cup(\partial\Omega\cap S)} w = \sup_{\Omega\cap S} |w|=1$.
    
    Let $\varphi(x)$ be the harmonic function
    \begin{equation*}
        \varphi(x)=(n-1)|x''-2/3|^2-|x'|^2+1,
    \end{equation*}
    then we have $\varphi\geq0$ in $\overline{S}$. Also let $\varphi_a(x)=a\cdot\varphi(x)$ for $a>0$. Let $a_0=\inf\{a>0:\varphi_a(x)\geq w(x) \text{ for } x\in \Gamma_+\cup(\partial\Omega\cap S)\}$. Since $\inf_{\Gamma_+\cup(\partial\Omega\cap S)} \varphi>0$ and $\sup_{\Gamma_+\cup(\partial\Omega\cap S)} w = 1$, the infimum exists. For the maximum point $\overline{x}$, we also have $\varphi_{a_0}(\overline{x})\geq w(\overline{x})$, which implies
    \begin{equation*}
        a_0\geq \frac{1}{(n-1)|\overline{x}''-2/3|^2-|\overline{x}'|^2+1}.
    \end{equation*}
    Since $\overline{x}\in\Gamma_+\cup(\partial\Omega\cap S)$, we have $|\overline{x}''-2/3|\leq 1$. Therefore $a_0\geq c(n)>0$. We have $\varphi_{a_0}(x)\geq w(x)$ on $\Gamma_+\cup(\partial\Omega\cap S)$ and $\varphi_{a_0}(x)\geq 0=w(x)$ on $\Sigma\cap\Omega$, by the maximum principle, $\varphi_{a_0}(x)\geq w(x)$ in $S\cap\Omega$. By the definition of $a_0$, there exists $y\in\Gamma_+\cup(\partial\Omega\cap S)$ such that $w(y)=\varphi_{a_0}(y)$. If $y\in\Gamma_+$, then 
    \begin{equation*}
        \frac{\partial w}{\partial e_n}(y)\geq\frac{\partial \varphi_{a_0}}{\partial e_n}(y)=a_0\cdot\frac{\partial \varphi}{\partial e_n}(y)=a_0\cdot 2(n-1)(y''-2/3)\geq c(n)>0,
    \end{equation*}
    since $y''-2/3=1/3$ for $y\in\Gamma_+$. If $y\in\partial\Omega\cap S$, then 
    \begin{equation*}
        \frac{\partial w}{\partial e_n}(y)\leq\frac{\partial \varphi_{a_0}}{\partial e_n}(y)=a_0\cdot\frac{\partial \varphi}{\partial e_n}(y)=a_0\cdot 2(n-1)(y''-2/3)\leq -c(n)<0,
    \end{equation*}
    since $y''-2/3\leq -1/3$ for $y\in\partial\Omega\cap S$. In both cases, we have
    \begin{equation*}
        \left| \frac{\partial w}{\partial e_n}(y) \right|\geq c(n)>0.
    \end{equation*}
    Then \eqref{ineq1} follows.
\end{proof}

Now we are ready to prove Theorem \ref{thm approximation}.

\begin{proof}[Proof of Theorem \ref{thm approximation}]
    For simplicity, we can choose the coordinate such that $x_1$ is the origin. Firstly, we extend $h$ to $\Sigma:=\{(x',x'')\in\partial S:0<x''<1\}$ continuously in the following way. For any $x\in\Sigma\cap\Omega$, define $\tilde{h}(x):=h(x)$. For any $(x',x'')\in\Sigma\setminus\Omega$, let $y_x''\in\R$ be the number such that $(x',y_x'')\in\Sigma\cap\partial\Omega$ and define $\tilde{h}(x',x''):=h(x',y_x'')$. Now we have $\tilde{h}\in C(\Sigma)$. Let $\Gamma_+:=\{(x',x'')\in\partial S:x''=1\}$ be the upper surface of $S$. Since $\Gamma_+\subset\Omega$, we have $h\in C^\infty(\Gamma_+)$. Let $g$ be the solution to the following problem (See \cite{lieberman1986mixed} for the existence of the solution):
    \begin{equation}\label{problem1}
        \left\{
        \begin{aligned}
            & \Delta g = 0 \ \text{in}\ S \\
            & g=\tilde{h} \ \text{on}\ \Sigma \\
            & \frac{\partial g}{\partial e_n} = \frac{\partial h}{\partial e_n} \ \text{on}\ \Gamma_+ \\
            & \frac{\partial g}{\partial e_n} = 0 \ \text{on}\ \Gamma_- \\
        \end{aligned}
        \right.
    \end{equation}
    where $\Gamma_-:=\{(x',x'')\in\partial S:x''=0\}$ is the lower surface of $S$. Let $w=g-h\in C(\overline{S\cap\Omega})$, then we have $\Delta w=0$ in $S\cap\Omega$, $w=0$ on $\Sigma\cap\Omega$ and $\frac{\partial w}{\partial e_n}=0$ on $\Gamma_+$.
    
    We want to use Lemma \ref{quantitative uniqueness of a mixed boundary value problem} to show that $\sup_{\Omega\cap S}|w|<C\tau$. Check the conditions, then it remains to show that $\left| \frac{\partial w}{\partial e_n} \right|<C\tau$ on $\partial\Omega\cap S$.

    Since we have normalized $\sup_{\Omega\cap B(x_0,3)}|h|=1$, 
    \begin{equation*}
        \sup_{\overline{\Omega}\cap B(0,2)}|\nabla h| \leq \sup_{\overline{\Omega}\cap B(x_0,2+3\tau)}|\nabla h| \leq \sup_{\overline{\Omega}\cap B(x_0,5/2)}|\nabla h| \leq C_1
    \end{equation*}
    by Lemma \ref{local gradient estimate} and $\tau<1/100$. On $\partial\Omega\cap B(0,2)$, we have
    \begin{equation*}
        \left|\frac{\partial h}{\partial e_n}\right| = |\nabla h \cdot (-e_n)| \leq |\nabla h\cdot \nu| + |\nabla h\cdot (-e_n-\nu)| \leq C |-e_n-\nu| \leq C_2 \tau,
    \end{equation*}
    since $\frac{\partial h}{\partial \nu}=0$ on $\partial\Omega\cap B(0,2)$ and $\partial\Omega\cap B(0,2)\in C^{1,1}(\tau)$ thus the unit vectors $-e_n$ and $\nu$ are close in terms of $\tau$.

    Let $\psi(x)$ be the harmonic function on $B_2=B(0,2)$ such that $\psi(x)=1$ on $\partial B_{2,+}:=\{(x',x'')\in \partial B_2: x''>0\}$ and $\psi(x)=-1$ on $\partial B_{2,-}:=\{(x',x'')\in \partial B_2: x''<0\}$. Obviously, $\psi(x)=0$ on $\{x''=0\}$. Define the linear function $l(x',x'')=x''/2$. By the maximum principle, we have $\psi(x)\geq l(x)$ in $B_{2,+}:=\{(x',x'')\in B_2:x''>0\}$. For any $(x',x'')\in\partial\Omega\cap B_2$, we have $x''\geq |x_0''|-|x''-x_0''|\geq 3\tau-2\tau=\tau$. Therefore $\psi(x)\geq l(x)\geq \tau/2$ on $\partial\Omega\cap B_2$. Let $\bar{\psi}=2C_3\psi$, where $C_3=\max\{C_1,C_2\}$. Then we have for any $x\in\partial\Omega\cap B_2$,
    \begin{equation*}
        \bar{\psi}(x)\geq C_2\tau \geq \left|\frac{\partial h}{\partial e_n}(x)\right|.
    \end{equation*}
    And for any $x\in\partial B_2\cap\Omega$, we have $\psi(x)=1$, then
    \begin{equation*}
        \bar{\psi}(x)\geq C_1 \geq \left|\frac{\partial h}{\partial e_n}(x)\right|.
    \end{equation*}
    Note that the function $\frac{\partial h}{\partial e_n}$ is also harmonic. By maximum principle, we have $\bar{\psi}(x)\geq\left|\frac{\partial h}{\partial e_n}(x)\right|$ for $x\in\Omega\cap B_2$. Particularly, $\bar{\psi}(x)\geq\left|\frac{\partial h}{\partial e_n}(x)\right|$ for $x\in\partial S\cap\Omega$.
    
    Recall that $g$ is defined by \eqref{problem1}, we have $\left|\frac{\partial g}{\partial e_n}(x)\right|=\left|\frac{\partial h}{\partial e_n}(x)\right|$ for $x\in\partial S\cap\Omega$ and $\left|\frac{\partial g}{\partial e_n}(x)\right|=0$ for $x\in\partial S\setminus\Omega$. Then we have for $x\in\partial S$, $\bar{\psi}(x)\geq\left|\frac{\partial g}{\partial e_n}(x)\right|$. Note that the function $\frac{\partial g}{\partial e_n}(x)$ is also harmonic. By the maximum principle, we have $\bar{\psi}(x)\geq\left|\frac{\partial g}{\partial e_n}(x)\right|$ for $x\in S$. On the other hand, by the Cauchy estimates, $\sup_S|\nabla \bar{\psi}|\leq C \sup_{B_2}|\bar{\psi}|\leq C$. Note that $\bar{\psi}=0$ on $\Gamma_-$, we have $\bar{\psi}(x',x'')\leq Cx''$ in $S$ and thus $\left|\frac{\partial g}{\partial e_n}(x',x'')\right|\leq Cx''$ in $S$. For $(x',x'')\in\partial\Omega\cap S$, we have $x''\leq 4\tau$. Therefore $\left|\frac{\partial g}{\partial e_n}(x)\right|\leq C\tau$ on $\partial\Omega\cap S$, which is the desired claim.

    For the second estimate in the theorem, we note that in our construction, $\frac{\partial g}{\partial e_n}=0$ on $\Gamma_-$. Then by the maximum principle and Hopf's lemma, $|g|$ can only attain its maximum on $\Gamma_+\cup\Sigma$. Again by the construction of $g$, we have $g=h$ on $\Gamma_+$ and $g=\tilde{h}$ on $\Sigma$. Therefore,
    \begin{equation*}
        \sup_S|g|=\sup_{\Gamma_+\cup\Sigma}|g|= \max\{\sup_{\Gamma_+}|h|,\sup_{\Sigma}|\tilde{h}|\} = \sup_{\Gamma_+\cup(\Sigma\cap\Omega)}|h|= \sup_{\Omega\cap S}|h|,
    \end{equation*}
    where we also use the maximum principle and Hopf's lemma for $h$ in the last equality.
\end{proof}

\end{section}

\begin{section}{Nodal Sets near the Boundary}\label{Nodal Sets near the Boundary}

\begin{subsection}{A Standard Division of Cubes on the Boundary}
    The following division process of cubes is firstly introduced by Logunov in \cite{logunov2021sharp}. It is worth pointing out that though we argue on $C^{1,1}(\tau)$ domains in this paper, the following division also works on Lipschitz domains with small Lipschitz constant.
    
    Let $\Omega$ be a bounded domain in $\R^n$ and $\partial\Omega\cap B\in C^{1,1}(\tau)$ where $B=B(x,r)$ for some $x\in\partial\Omega$. Also assume $\tau<(16\sqrt{n})^{-1}$. Choose the local coordinate as in Remark \ref{choice of the coordinate}. Let $Q$ be a cube centred at $x_Q\in\partial\Omega\cap B$ whose sides are parallel to the axes of the coordinate system and such that $Q\subset B$. Denote the side length of $Q$ by $s$.

    We cover $\Omega\cap Q$ by small cubes in the following way. Firstly, consider the projection $\pi(Q)$ of $Q$ to $\R^{n-1}\times \{0\}$. Let $k\geq 3$. Divide $\pi(Q)$ into $2^{k(n-1)}$ equal small cubes $w$ with side length $2^{-k}s$ in the normal way. For each $(n-1)$-dimensional $w$, there exists a unique $n$-dimensional cube $q$ such that $\pi(q)=w$ and the center of $q$ lies on $\partial\Omega\cap B$. Then we use at most $2^k$ small cubes $p$ with side length $2^{-k}s$ to cover $(\pi^{-1}(w)\cap(\Omega\cap Q))\setminus q$ such that $p\subset Q$, $\pi(p)=w$, $p$ has no inner common points with $q$. Cubes $p$ may overlap.

    Let $\mathcal{B}_k(Q)$ denote the set of cubes $q$ which intersect the boundary and $\mathcal{I}_k(Q)$ denote the set of inner cubes $p$. Since $\tau<(16\sqrt{n})^{-1}$, for inner cubes $p$, $dist(p,\partial\Omega)\geq s/4$. The triple $(Q,\mathcal{B}_k(Q),\mathcal{I}_k(Q))$ is called \textit{the standard construction}.
\end{subsection}

\begin{subsection}{Cubes with Large Doubling Indexes}

\begin{lemma}\label{cubes with large doubling}
    There exist constants $k_0\geq 3$ and $N_0\geq 1$ such that for any integer $k\geq k_0$, there exists $\tau(k)>0$ for which the following statement holds. Suppose that $\Omega$ is a $C^{1,1}$ domain in $\R^n$, $\partial\Omega\cap B\in Lip(\tau)$, $\tau<\tau(k)$, and $Q\subset\frac{1}{64}B$ is a cube centred on $\partial\Omega$. Then for any $h\in C(\overline{\Omega})$ harmonic in $\Omega$, with $\frac{\partial h}{\partial\nu}=0$ on $\partial\Omega\cap B$, and $N^*_h(Q)>N_0$, there exists a cube $q\in\mathcal{B}_k(Q)$ such that $N^*_h(q)\leq N^*_h(Q)/2$.
\end{lemma}

\begin{proof}
    Let $x_Q$ and $l$ denote the center of $Q$ and the diameter of $Q$ respectively. Let $B_1=B(x_Q,l)$. Then $B_1\subset B$. Denote $M^2=\int_{B_1\cap\Omega}h^2$ and $N=N^*_h(Q)$.

    We prove the lemma by contradiction. Suppose $N^*_h(q)>N/2$ for each $q\in\mathcal{B}_k(Q)$. Then for each $q$, there exists $y_q\in q\cap\overline{\Omega}$ and $0<r_q\leq2^{-k}l$ such that $N_h(y_q,r_q)>N/2$. With the assumption that $N$ is sufficiently large, the almost monotonicity of the doubling index (Lemma \ref{monotonicity of doubling}) implies $N_h(y_q,2^j r_q)>cN$ for $1\leq j\leq k$ and $c$ only depneds on the $C^{1,1}$ character of $\partial\Omega$.

    Assumeing that $k\geq 20$, apply the estimate of the doubling index $k-4$ times and use that $B(y_q,l/2)\subset B_1$ to obtain
    \begin{equation*}
        \begin{split}
            \int_{B(y_q,2^{-k+2}l)\cap\Omega} h^2 &\leq \int_{B(y_q,8r_q)\cap\Omega} h^2 \leq e^{-cN(k-4)}\int_{B(y_q,2^{k-1}r_q)\cap\Omega} h^2 \\
            &\leq e^{-cN(k-4)}\int_{B(y_q,l/2)\cap\Omega} h^2 \leq e^{-c'Nk} M^2,
        \end{split}
    \end{equation*}
    By the local boundedness near the boundary (Lemma \ref{local boundedness}), we have
    \begin{equation}\label{bound of h}
        \sup_{B(y_q,2^{-k+1}l)\cap\Omega}h^2 \leq C2^{nk}l^{-n} \int_{B(y_q,2^{-k+2}l)\cap\Omega} h^2 \leq C2^{nk}l^{-n}e^{-cNk}M^2.
    \end{equation}

    As in the standard construction, assume that $\tau<(16\sqrt{n})^{-1}$. For each $q\in\mathcal{B}_k(Q)$, denote by $q^+$ its upper quarter, where ``up" is in the direction of $e_n$. Then $dist(q^+,\partial\Omega)\geq 2^{-k-3}s$, where $s$ denotes the side length of $Q$. Thus we have $l=\sqrt{n}s$. We also have for any $y\in q^+$, $B(y,2^{-k-3}s)\subset B(y_q,2^{-k+1}l)$. By the standard Cauchy estimate,
    \begin{equation}\label{bound of Dh}
        \sup_{q^+}|\nabla h|\leq C2^kl^{-1}\sup_{B(y_q,2^{-k+1}l)\cap\Omega} |h| \leq C2^{k(n+2)/2}l^{-(n+2)/2}e^{-cNk}M.
    \end{equation}
    
    Let $B_0=B(x_Q+3\cdot2^{-k-3}se_n,s/2)$ and let $B_{0,+}=\{(x',x'')\in B_0:x''\geq x_Q+3\cdot2^{-k-3}s\}$ be the upper half of $B_0$. We denote by $\Gamma_0$ the flat part of the boundary of $B_{0,+}$. Note that $2B_0\subset B_1$. Assuming that $\tau<2^{-k-3}$, we have $dist(B_{0,+},\partial\Omega)\geq 2^{-k-2}s$. Then by the local boundedness (Lemma \ref{local boundedness}) and the interior Cauchy estimate, we get
    \begin{equation*}
        \sup_{B_0\cap\Omega} |h| \leq Cl^{-n/2} M,\quad \sup_{B_{0,+}}|\nabla h| \leq C 2^k l^{-(n+2)/2} M.
    \end{equation*}
    We also have $\Gamma_0\subset\bigcup_{q\in\mathcal{B}_k(Q)}q^+$ since $\tau<2^{-k-3}$. Then by \eqref{bound of h} and \eqref{bound of Dh}, we have 
    \begin{equation*}
        \sup_{\Gamma_0}|h|\leq C2^{kn/2}l^{-n/2}e^{-cNk}M,\quad \sup_{\Gamma_0}|\nabla h| \leq C2^{k(n+2)/2}l^{-(n+2)/2}e^{-cNk}M.
    \end{equation*}
    Let $\tilde{h}=C^{-1}2^{-k}l^{n/2}M^{-1}h$, then $\tilde{h}$ satisfies
    \begin{equation*}
        \sup_{B_0\cap\Omega} |\tilde{h}| \leq 1,\quad \sup_{B_{0,+}}|\nabla \tilde{h}| \leq 1,\quad \sup_{\Gamma_0}|\tilde{h}|\leq \varepsilon,\quad \sup_{\Gamma_0}|\nabla\tilde{h}|\leq \varepsilon,
    \end{equation*}
    where $\varepsilon=C2^{kn/2}e^{-cNk}$. Apply Lemma \ref{propagation of smallness} to $\tilde{h}$, we have $\sup_{\frac{1}{3}B_{0,+}}|\tilde{h}|\leq C\varepsilon^\gamma$, that is
    \begin{equation*}
        \sup_{\frac{1}{3}B_{0,+}} |h| \leq C 2^{\gamma kn/2+k}l^{-n/2}e^{-c\gamma Nk}M.
    \end{equation*}

    Let $y_Q=x_Q+\frac{s}{12}e_n$ and let $m$ be the least interger such that $2^m>16\sqrt{n}$. Then $B_2=B(y_Q,2^{-m}l)\subset\frac{1}{3}B_{0,+}$ for sufficiently large $k$. With the fact $|B_2|\leq Cl^n$, we have
    \begin{equation*}
        \int_{B_2}h^2 \leq |B_2|\sup_{B_2} h^2 \leq |B_2|\sup_{\frac{1}{3}B_{0,+}} h^2 \leq C2^{\gamma kn+2k}e^{-c\gamma Nk}M^2.
    \end{equation*}

    Finally, we estimate the growth of $h$ at $y_Q$ in terms of the doubling index. On the one hand, since $B_1\subset B(y_Q,2l)$,
    \begin{equation*}
        \log\frac{\int_{B(y_Q,2l)\cap\Omega}h^2}{\int_{B_2}h^2} \geq \log\frac{\int_{B_1\cap\Omega}h^2}{\int_{B_2}h^2}\geq c\gamma Nk-\gamma kn -2k-C,
    \end{equation*}
    where we have used $\log2<1$. On the other hand, by the almost monotonicity of the doubling index (Lemma \ref{monotonicity of doubling}),
    \begin{equation*}
        \log\frac{\int_{B(y_Q,2l)\cap\Omega}h^2}{\int_{B_2}h^2} = \sum_{j=0}^m N_h(y_Q,2^{-j}l) \leq C(m+1)(N_h(y_Q,l)+1)\leq C'(N+1).
    \end{equation*}
    Then the inequality above implies
    \begin{equation*}
        c\gamma Nk-\gamma kn -2k-C \leq C'(N+1).
    \end{equation*}
    Choose $k$ sufficiently large so that $c\gamma k/2>C'$, then we have
    \begin{equation*}
        N\leq \frac{2(\gamma kn+2k+C)}{c\gamma k}\leq C(n+\gamma^{-1}). 
    \end{equation*}
    Taking $N_0=C(n+\gamma^{-1})$, we obtain a contradiction for $N>N_0$.
\end{proof}
    
\end{subsection}

\begin{subsection}{Cubes with Small Doubling indexes}\label{Cubes with Small Doubling indexes}

\begin{lemma}\label{cubes with small doubling}
    For any $N>0$ there exist $\tau(N)$ and $k(N)$ such that the following statement holds. Suppose that $\Omega$ is a $C^{1,1}$ domain in $\R^n$, $\partial\Omega\cap B\in Lip(\tau)$, $\tau<\tau(N)$, and $Q\subset\frac{1}{64}B$ is a cube centred on $\partial\Omega$. Let $h\in C(\overline{\Omega})$ be a non-zero function harmonic in $\Omega$, with $\frac{\partial h}{\partial \nu}=0$ on $\partial\Omega\cap B$ and $N^*_h(Q)\leq N$. Then for any $k\geq k(N)$, there exists $q\in\mathcal{B}_k(Q)$ such that $Z(h)\cap q=\varnothing$.
\end{lemma}

Considering the argument in Section \ref{Approximation for the Neumann Problem}, we first prove the following version of the lemma for a half ball.

\begin{lemma}\label{lemma for a half ball}
    For any $N>0$, there exist $\rho=\rho(N)$ and $c_0=c_0(N)$ such that the following statement holds. Let $B$ be the unit ball in $\R^n$ and let $B_+$ be the upper half ball, $B_+=\{y\in\R^n:|y'|^2+|y''|^2<1,y''>0\}$. Let $g\in C(\overline{B_+})$ be a harmonic function in $B_+$, $\frac{\partial g}{\partial\nu}=0$ on $\Gamma_0=\overline{B_+}\cap\{y''=0\}$, and $g$ is normalized as $\sup_{\frac{1}{4}B_+}|g|=1$. If $N_g(0,1/4)\leq N$, then there exists $x\in\Gamma_0\cap\frac{1}{8}B$ such that
    \begin{equation*}
        |g(y)| \geq c_0,
    \end{equation*}
    for any $y\in B(x,\rho)\cap B_+$.
\end{lemma}

\begin{proof}
    Let $B_-$ be the reflection of $B_+$ with respect to the hyperplane $y''=0$. Then $g$ can be extended to a harmonic function in $B$ by $g(y',y'')=g(y',-y'')$ when $(y',y'')\in B_-$. We denote this extension by $g$ as well. Since $\sup_{\frac{1}{4}B}|g|=1$, by the Cauchy estimate,
    \begin{equation*}
        \sup_{\frac{1}{8}B}|\nabla g| \leq C\sup_{\frac{1}{4}B}|g| \leq C.
    \end{equation*}

    Let $\delta=\sup_{\Gamma_0\cap\frac{1}{16}B}|g|$. Since $\frac{\partial g}{\partial\nu}=0$ on $\Gamma_0$, by the simple fact that $L^2$-norm is controlled by $L^\infty$-norm and the propagation of smallness (Lemma \ref{propagation of smallness}), we have
    \begin{equation*}
        \int_{\frac{1}{64}B}g^2 \leq C\sup_{\frac{1}{64}B} g^2 \leq C\delta^{2\gamma}.
    \end{equation*}
    By the local boundedness (Lemma \ref{local boundedness}), we also have
    \begin{equation*}
        \int_{\frac{1}{2}B}g^2 \geq c\sup_{\frac{1}{4}B}g^2 = c.
    \end{equation*}
    Then we have
    \begin{equation*}
        \frac{\int_{\frac{1}{2}B}g^2}{\int_{\frac{1}{64}B}g^2} \geq c\delta^{-2\gamma}.
    \end{equation*}
    On the other hand, we have
    \begin{equation*}
        \log\frac{\int_{\frac{1}{2}B}g^2}{\int_{\frac{1}{64}B}g^2} \leq 5N_g(0,1/4) \leq 5N,
    \end{equation*}
    where we have used the monotonicity of the doubling index \eqref{monotonicity of doubling interior}. We conclude that $\delta\geq ce^{-3N\gamma^{-1}}$. Then there exists $x\in\Gamma_0\cap\frac{1}{16}\overline{B}$ such that $|g(x)|\geq ce^{-3N\gamma^{-1}}$. Since the gradient of $g$ in $\frac{1}{8}B$ is bounded by a constant, there exists $\rho=\rho(N)$ such that
    \begin{equation*}
        |g(y)|\geq \frac{c}{2}e^{-3N\gamma^{-1}},
    \end{equation*}
    for any $y\in B(x,\rho)$.
\end{proof}

\begin{proof}[Proof of Lemma \ref{cubes with small doubling}]
    By rescaling, we can assume that the side length of $Q$, $s(Q)=4$. We may also assume that 
    \begin{equation*}
        \sup_{B(x_Q,3)\cap\Omega}|h|=1.
    \end{equation*}
    
    Let $x_1=x_Q-3\tau e_n$, $B_1=B(x_1,1)$, $B_{1,+}$ be the upper half of $B_1$ and $\Gamma_0$ be the flat part of $\partial B_{1,+}$. Consider the standard approximation defined in Section \ref{Approximation for the Neumann Problem}, $(x_1,S,\Omega\cap S)$. By Theorem \ref{thm approximation}, for $\tau<1/100$, there exists $g\in C(\overline{S})$ such that $\Delta g=0$ in $S$, $\frac{\partial g}{\partial e_n}=0$ on $\Gamma_-$, the lower surface of $S$ and
    \begin{equation}\label{difference between g and h}
        \sup_{\Omega\cap S}|g-h|<C\tau,
    \end{equation}
    \begin{equation}
        \sup_S |g| = \sup_{\Omega\cap S}|h|.
    \end{equation}

    The assumption $N^*_h(Q)\leq N$ implies $N_h(x_Q,4)\leq N$. By the almost monotonicity of the doubling index (Lemma \ref{monotonicity of doubling}) and the local boundedness (Lemma \ref{local boundedness}),
    \begin{equation*}
        \int_{B(x_Q,\frac{1}{8})\cap\Omega}h^2 \geq e^{-C(N+1)} \int_{B(x_Q,4)\cap\Omega}h^2 \geq ce^{-C(N+1)}\sup_{B(x_Q,3)\cap\Omega}|h| = ce^{-C(N+1)}.
    \end{equation*}
    Recall that $\tau<1/100$, we have $B(x_Q,\frac{1}{8})\cap\Omega\subset\frac{1}{4}B_{1,+}$. Then
    \begin{equation*}
        \left( \int_{\frac{1}{4}B_{1,+}}g^2 \right)^{\frac{1}{2}} \geq \left( \int_{B(x_Q,\frac{1}{8})\cap\Omega}g^2 \right)^{\frac{1}{2}} \geq \left( \int_{B(x_Q,\frac{1}{8})\cap\Omega}h^2 \right)^{\frac{1}{2}} - C\tau \geq ce^{-C(N+1)}-C\tau.
    \end{equation*}
    Assuming that $\tau(N)$ is sufficiently small, we conclude that
    \begin{equation}\label{lower bound of g}
        \int_{\frac{1}{4}B_{1,+}}g^2 \geq ce^{-C(N+1)}.
    \end{equation}
    We also have
    \begin{equation*}
        \int_{\frac{1}{2}B_{1,+}}g^2 \leq C\sup_{\frac{1}{2}B_{1,+}}g^2 \leq C\sup_S g^2 = C\sup_{\Omega\cap S} h^2 \leq C\sup_{B(x_Q,3)\cap\Omega} h^2 = C.
    \end{equation*}
    Then we conclude that
    \begin{equation*}
        N_g(x_1,\frac{1}{4}) = \log\frac{\int_{\frac{1}{2}B_{1,+}}g^2}{\int_{\frac{1}{4}B_{1,+}}g^2} \leq C(N+1).
    \end{equation*}
    By \eqref{lower bound of g}, we have $\sup_{\frac{1}{4}B_{1,+}}|g|\geq ce^{-C(N+1)}$. We apply Lemma \ref{lemma for a half ball} to $g$ and conclude that there exist $x_*\in\Gamma_0\cap\frac{1}{8}B_1$, $\rho=\rho(N)$ and $c_0=c_0(N)$ such that for any $y\in B(x_*,\rho)$,
    \begin{equation*}
        |g(y)|\geq c_0e^{-C(N+1)}.
    \end{equation*}

    Let $\tau(N)<\frac{1}{100}\rho(N)$ and $k(N)$ be sufficiently large, then for $k>k(N)$, at least one small cube $q\in\mathcal{B}_k(Q)$ satisfies $q\cap\Omega\subset B(x_*,\rho)\cap\Omega$. Also let $\tau(N)$ be sufficiently small in \eqref{difference between g and h} so that $\sup_{\Omega\cap S}|g-h|<\frac{c_0}{2}e^{-C(N+1)}$. Then we have for any $y\in q\cap\Omega$,
    \begin{equation*}
        |h(y)|\geq|g(y)|-\frac{c_0}{2}e^{-C(N+1)}\geq \frac{c_0}{2}e^{-C(N+1)}.
    \end{equation*}
    Thus $Z(h)\cap q=\varnothing$.
\end{proof}

\end{subsection}

\end{section}

\begin{section}{Proof of Theorem \ref{nodal set bounds for harmonic functions}}\label{Proof of Theorem}

Let $k_0=k_0(n)$ and $N_0=N_0(n)$ be as in Lemma \ref{cubes with large doubling}. For this $N_0$, by Lemma \ref{cubes with small doubling} , there exist relative constants $\tau(n)$ and $k_1(n)$. Let $k_2(n)=\max\{k_0,k_1\}$. Let $B=B(x,128r)$ with $x\in\partial\Omega$ as in the statement of Theorem \ref{nodal set bounds for harmonic functions} and $Q\subset 2B$ centred on $\partial\Omega$. We also write $N^{**}_h(Q)=\max\{N^*_h(Q),N_0/2\}$. Consider the standard construction $(Q,\mathcal{B}_k(Q),\mathcal{I}_k(Q))$, then we have for $k\geq k_2(n)$, $\tau\leq\tau(n)$, there exists $q_0\in\mathcal{B}_k(Q)$ such that at least one of the following statements holds:
\begin{equation*}
    \text{(i) } N^{**}_h(q_0)\leq N^{**}_h(Q)/2,\quad \text{(ii) } Z(h)\cap q_0=\varnothing.
\end{equation*}
\begin{subsection}{Reduction to boundary cubes}

Let $Q$ centred on $\partial\Omega$ as above, we claim that
\begin{equation}\label{nodal volume bounds for boundary cubes}
    \mathcal{H}^{n-1}(Z(h)\cap Q) \leq C N^{**}_h(Q) s(Q)^{n-1}.
\end{equation}

Assume \eqref{nodal volume bounds for boundary cubes} holds, we show that Theorem \ref{nodal set bounds for harmonic functions} follows.

We cover the ball $B(x,r)$ with cubes $Q_j\subset B(x,2r)$ such that $diam(Q_j)=r/10$ and either $dist(Q_j,\partial\Omega)>s(Q_j)/10$ (inner cubes) or $Q_j$ centred on $\partial\Omega$ (boundary cubes). The number of such cubes does not exceed $C=C(n)$.

For each $Q=Q_j$ in this covering, by the definition of $N^*_h(Q)$, there exist $y\in Q\cap\overline{\Omega}$ and $0<r_y\leq r/10$ such that $N_h(y,r_y)=N^*_h(Q)$.  By the almost monotonicity of the doubling index (Lemma \ref{monotonicity of doubling}), we have $N_h(y,2r)\geq c N_h(y,r_y)-C$. On the other hand, since $dist(x,y)\leq \frac{11}{10}r$, we have
\begin{equation*}
    N_h(y,2r)=\log\frac{\int_{B(y,4r)\cap\Omega}h^2}{\int_{B(y,2r)\cap\Omega}h^2} \leq \log\frac{\int_{B(x,8r)\cap\Omega}h^2}{\int_{B(x,r/2)\cap\Omega}h^2} \leq C(N_h(x,4r)+1)
\end{equation*}
again by Lemma \ref{monotonicity of doubling}. Hence $N^*_h(Q)\leq C(N_h(x,4r)+1)$ and $N^{**}_h(Q)\leq C'(N_h(x,4r)+1)$.

Each inner cube $Q\subset\Omega$ can be covered by at most $C$ balls $B_j=B(y_j,s(Q)/100)$ with $y_j\in Q$. Then $8\overline{B_j}\subset\Omega$ and $N_h(y_j,s(Q)/25)\leq N^*_h(Q)$. By Lemma \ref{nodal set bounds for interior balls}, we have
\begin{equation*}
\begin{split}
    \mathcal{H}^{n-1} (Z(h)\cap Q) &\leq \sum_j \mathcal{H}^{n-1} (Z(h)\cap B_j) 
    \leq C(N_h(y_j,s(Q)/25)+1)r^{n-1} \\
    &\leq C'(N^*_h(Q)+1)r^{n-1} \leq C''(N_h(x,4r)+1)r^{n-1}.
\end{split}
\end{equation*}

For boundary cubes, we use \eqref{nodal volume bounds for boundary cubes}. Thus for each boundary cube $Q$,
\begin{equation}
    \mathcal{H}^{n-1}(Z(h)\cap Q) \leq C N^{**}_h(Q) s(Q)^{n-1} \leq C'(N_h(x,4r)+1)r^{n-1}.
\end{equation}
Summing these inequalities over all cubes, we obtain the required estimate. It remains to prove \eqref{nodal volume bounds for boundary cubes}.
\end{subsection}

\begin{subsection}{Proof of (\ref{nodal volume bounds for boundary cubes})}

We fix a compact set $K\subset\Omega$ and prove that
\begin{equation}\label{nodal volume bounds for any compact set}
    \mathcal{H}^{n-1}(Z(h)\cap Q\cap K) \leq C_0 N^{**}_h(Q) s(Q)^{n-1},
\end{equation}
where $Q\subset2B$ is a cube  as in the standard construction and $C_0$ is independent of $K$. Then \eqref{nodal volume bounds for boundary cubes} follows.

For cubes $Q$ with sufficiently small side length, \eqref{nodal volume bounds for any compact set} holds since $Q\cap K=\varnothing$. We prove \eqref{nodal volume bounds for any compact set} by induction. Assume it holds for cubes with $s(Q)<s$, we prove it for cubes with $s(Q)<2^k s$, where $k\geq k_2(n)$ as in the beginning of this section.

We consider the standard construction $(Q,\mathcal{B}_k(Q),\mathcal{I}_k(Q)$. Each inner cube $q\in\mathcal{I}_k(Q)$ can be covered by balls $b$ centred in $q\subset Q$ with radii $s(q)/100$ and such that $8\bar{b}\subset\Omega$, so that the number of balls is bounded by a dimensional constant. For each such ball $b=B(y,s(q)/100)$, we have $N_h(y,s(q)/25)\leq N^*_h(Q)$. Then by Lemma \ref{nodal set bounds for interior balls}, we have
\begin{equation*}
    \sum_{q\in\mathcal{I}_k(Q)}\mathcal{H}^{n-1}(Z(h)\cap q) \leq C (N^*_h(Q)+1) s(Q)^{n-1}\leq C_1N^{**}_h(Q) s(Q)^{n-1},
\end{equation*}
where $C_1$ depends on $k$ and thus only on $n$.

For boundary cubes $q$, we have $N^{**}_h(q)\leq N^{**}_h(Q)$ since $q\subset Q$. Recall that there exists a cube $q_0\in\mathcal{B}_k(Q)$ such that either $N^{**}_h(q_0)\leq N^{**}_h(Q)/2$ or $Z(h)\cap q_0=\varnothing$. We use the induction assumption for every boundary cube and obtain
\begin{equation*}
    \begin{split}
        \mathcal{H}^{n-1} &(Z(h)\cap K\cap(\cup_{q\in\mathcal{B}_k(Q)}q)) \\
        &\leq \sum_{q\in\mathcal{B}_k(Q),q\neq q_0} \mathcal{H}^{n-1} (Z(h)\cap K\cap q) + \mathcal{H}^{n-1} (Z(h)\cap K\cap q_0) \\
        &\leq \sum_{q\in\mathcal{B}_k(Q),q\neq q_0} C_0 N^{**}_h(q) s(q)^{n-1} + \frac{C_0}{2} N^{**}_h(Q) s(q_0)^{n-1} \\
        &\leq \left( \frac{2^{k(n-1)}-1}{2^{k(n-1)}} + \frac{1}{2}\cdot\frac{1}{2^{k(n-1)}} \right) C_0 N^{**}_h(Q) s(Q)^{n-1}.
    \end{split}
\end{equation*}
Finally, we choose $C_0$ large enough so that
\begin{equation*}
    C_1 + \left( \frac{2^{k(n-1)}-1}{2^{k(n-1)}} + \frac{1}{2}\cdot\frac{1}{2^{k(n-1)}} \right) C_0 < C_0.
\end{equation*}
Note that $C_0$ does not depend on $K$. Then, summing the interior and boundary estimates, we obtain
\begin{equation*}
    \mathcal{H}^{n-1}(Z(h)\cap K\cap Q) \leq C_0 N^{**}_h(Q) s(Q)^{n-1}.
\end{equation*}
This concludes the induction step and the proof of \eqref{nodal volume bounds for boundary cubes}.

\end{subsection}

\end{section}

\begin{section}{Neumann Laplace Eigenfunctions}\label{Neumann Laplace Eigenfunctions}

\begin{subsection}{Harmonic extension and an estimate on the doubling index}

For the Neumann eigenfunction $u$ satisfying $\Delta u + \lambda u = 0$ in $\Omega_0\subset\R^n$ and $\frac{\partial u}{\partial \nu}=0$ on $\partial\Omega_0$. Let $\Omega=\Omega_0\times\R$ and 
\begin{equation}\label{harmonic extension}
    h(x,t)=e^{\sqrt{\lambda}t}u(x).
\end{equation}
Then $\Delta h=0$ in $\Omega$ and $\frac{\partial h}{\partial \nu}=0$ on $\partial\Omega=\partial\Omega_0\times\R$. And the zero set $Z(h)=Z(u)\times\R$. Also, we have the following estimate on the doubling index.

\begin{lemma}\label{estimate on doubling index}
    Let $\Omega_0\subset\R^n$ be a bounded $C^{1,1}$ domain. There exists $r_0=r_0(\Omega_0)$  such that the following statement holds. For any $r<r_0/16$ there exists $C=C(r,\Omega_0)$  such that for any Neumann Laplace eigenfunction $u$ in $\Omega_0$ with eigenvalue $\lambda$, the corresponding harmonic extension $h$ satisfies $N_h(y,r)\leq C\sqrt{\lambda}$ for any $y\in\overline{\Omega}=\overline{\Omega}_0\times\R$. 
\end{lemma}

\begin{proof}
    Let $r_0$ be such that $\partial\Omega\cap B(x,r_0)$ can be represented by a $C^{1,1}$ function in local coordinates for any $x\in\partial\Omega$. We have $r_0=r_0(\Omega)$. For any $y=(x,t)\in\overline{\Omega}$, let $y_0=(x,0)$. Since $h(x,t+s)=e^{\sqrt{\lambda}t}h(x,s)$, we have $N_h(y,r)=N_h(y_0,r)$. So it is sufficient to prove the estimate for $y\in\Omega\times\{0\}$.

    Fix $r<r_0/16$ and let $\mathcal{S}\subset\overline{\Omega}_0$ be a $r/8$-net for $\overline{\Omega}_0$, i.e., $\overline{\Omega}_0\subset\bigcup_{p\in\mathcal{S}}B(p,r/8)$. The number of points in $\mathcal{S}$ does not exceed $C(n,r,diam(\Omega_0))$. Assume that $|u(x_0)|=\max_{\Omega_0}|u|=1$. For any $y=(y',0)\in\overline{\Omega}_0\times\{0\}$, we consider a path $\gamma:[0,1]\to\overline{\Omega}_0$ from $y'$ to $x_0$. such that $\gamma((0,1))\subset\Omega_0$. Now we construct an overlapping chain of balls $\{B_j)\}_{j=0}^J$ with radii $r/2$. Let $y_0=y'$. Assuming that $y_j$ is constructed, we define $s_j=\sup\{ s\in[0,1]:|\gamma(s)-y_j|\leq r/8 \}$. If $s_j<1$, we have $|\gamma(s_j)-y_j|=r/8$ and choose $y_{j+1}\in\mathcal{S}$ such that $|\gamma(s_j)-y_{j+1}|<r/8$. If $s_j=1$, we choose $y_{j+1}=y_J=x_0$ and stop the chain. Define $B_j=B(y_j,r/2)$. In this way, we have $|y_j-y_{j+1}|<r/4$ for $0\leq j\leq J-1$ and $B_{j+1}\subset\frac{3}{2}B_j$. Also, $J$ does not exceed the number elements in $\mathcal{S}$ plus two.

    Let now $\tilde{B}_j=B((y_j,0),r/2)$ be the corresponding ball in $\R^{n+1}$. Then we have $\sup_{4\tilde{B}_j\cap\Omega}|h|\leq e^{2\sqrt{\lambda}r}$. By Corollary \ref{three ball ineq}, in both cases that $4\tilde{B}_j\subset\Omega$ and $4\tilde{B}_j\cap\partial\Omega\neq\varnothing$, we always have
    \begin{equation*}
        \sup_{\frac{3}{2}\tilde{B}_j\cap\Omega}|h| \leq C \left( \sup_{\tilde{B}_j\cap\Omega}|h| \right)^{1-\delta} \left( \sup_{4\tilde{B}_j\cap\Omega}|h| \right)^\delta \leq C e^{2\delta\sqrt{\lambda}r} \left( \sup_{\tilde{B}_j\cap\Omega}|h| \right)^{1-\delta},
    \end{equation*}
    where $C$ and $\delta$ only depends on $n$ and the $C^{1,1}$ character of $\Omega$. Therefore we obtain
    \begin{equation*}
        \sup_{\tilde{B}_j\cap\Omega}|h| \geq ce^{-\frac{2\delta}{1-\delta}\sqrt{\lambda}r} \left( \sup_{\frac{3}{2}\tilde{B}_j\cap\Omega}|h| \right)^{\frac{1}{1-\delta}} \geq ce^{-\frac{2\delta}{1-\delta}\sqrt{\lambda}r} \left( \sup_{\tilde{B}_{j+1}\cap\Omega}|h| \right)^{\frac{1}{1-\delta}}
    \end{equation*}
    Remind that we have $\sup_{\tilde{B}_J\cap\Omega}|h|=e^{\sqrt{\lambda}r/2}$. Combining the inequalities above, we obtain
    \begin{equation*}
        \sup_{\tilde{B}_0\cap\Omega}|h| \geq c_1e^{-C_2\sqrt{\lambda}},
    \end{equation*}
    where $c_1$ and $C_2$ depend on $r$ and $J$ but not on $\lambda$. The $r/8$-net $\mathcal{S}$ can be chosen so that the number of elements in $\mathcal{S}$ depends only on $diam(\Omega_0)$, $r$ and the dimension $n$. Thus $c_1$ and $C_2$ depend only on $diam(\Omega_0)$, $r$ and $n$.

    Finally, by Lemma \ref{local boundedness}, we have
    \begin{equation*}
        \begin{split}
            N_h(y,r) &= \log\frac{\int_{4\tilde{B}_0\cap\Omega} h^2}{\int_{2\tilde{B}_0\cap\Omega} h^2} \\
            &\leq \log\frac{\sup_{4\tilde{B}_0\cap\Omega} h^2}{\sup_{\tilde{B}_0\cap\Omega} h^2} + C\\
            &\leq (2r+C_2)\sqrt{\lambda}+C \leq C\sqrt{\lambda},
        \end{split}
    \end{equation*}
    where $C=C(\Omega_0,r)$. We also use the fact that $\lambda\geq\lambda_1(\Omega_0)>0$, where $\lambda_1(\Omega_0)$ is the first Neumann Laplace eigenvalue in $\Omega_0$.
\end{proof}

\end{subsection}

\begin{subsection}{Proof of Theorem \ref{main thm}}

Let $\Omega_0\subset\R^n$ be a bounded $C^{1,1}$ domain. Let $\tau_n$ be the constant as in Theorem \ref{nodal set bounds for harmonic functions}. For this $\tau_n$, there exists $r_0=r_0(\Omega_0)$ such that $\partial\Omega_0\cap B(x,r_0)\in Lip(\tau)$ for every $x\in\partial\Omega_0$. Consider the domain $\Omega=\Omega_0\times\R\subset\R^{n+1}$ and let $\Omega_1=\Omega_0\times[-1,1]$. Firstly, we cover the $2^{-12}r_0$-neighborhood of $\partial\Omega\times[-1,1]$ by balls centred on $\partial\Omega$ with radii $2^{-10}r_0$. Denote these balls by $\{B_j\}_{j=1}^J$. We have $\partial\Omega\cap128B_j\in Lip(\tau_n)$. The number $J$ depends only on $\partial\Omega_0$ and $r_0$ thus only on $\Omega_0$. Then we cover $\Omega_1\setminus(\cup_j B_j)$ by balls with centers in $\Omega_1\setminus(\cup_j B_j)$ and radii $2^{-15}r_0$. Denote these balls by $\{B'_k\}_{k=1}^K$. We have $8B'_k\subset\Omega$. The number $K$ depends only on $\Omega_0$.

Let now $u$ be a Neumann Laplace eigenfunction in $\Omega_0$: $\Delta u + \lambda u = 0$ in $\Omega_0$ and $\frac{\partial u}{\partial\nu}=0$ on $\partial\Omega_0$. Its harmonic extension $h(x,t)=e^{\sqrt{\lambda}t}u(x)$ satisfies $\frac{\partial h}{\partial\nu}$ on $\partial\Omega$. Let 
\begin{equation*}
    C_0=\max\{ C(2^{-8}r_0,\Omega_0), C(2^{13}r_0,\Omega_0) \},
\end{equation*}
where $C(r,\Omega_0)$ is as in Lemma \ref{estimate on doubling index}. Then for $B(x,r)\in\{B_j\}\cup\{B'_k\}$, we have $N_h(x,4r)\leq C_0\sqrt{\lambda}$. We apply Theorem \ref{nodal set bounds for harmonic functions} to the boundary balls $B_j$ and Lemma \ref{nodal set bounds for interior balls} to the inner balls $B'_k$ and conclude that
\begin{equation*}
    \begin{split}
        \mathcal{H}^n(Z(h)\cap\Omega_1) &\leq \sum_{j=1}^J \mathcal{H}^n(Z(h)\cap B_j) + \sum_{k=1}^K \mathcal{H}^n(Z(h)\cap B'_k)\\
        &\leq C(J,K)(C_0\sqrt{\lambda}+1)(r(B_j)^n+r(B'_k)^n) \leq C_1\sqrt{\lambda}.
    \end{split}
\end{equation*}
Then $\mathcal{H}^{n-1}(Z(u)\cap\Omega_0)\leq C_1\sqrt{\lambda}$, which finish the proof of Theorem \ref{main thm}.

\end{subsection}

\end{section}

\end{document}